\documentclass[12pt]{article}

\usepackage[margin=1in]{geometry}  
\usepackage{amsmath}               
\usepackage{amsthm}                
\usepackage{amssymb}
\usepackage[colorlinks=true]{hyperref}
\usepackage{mathabx}
\usepackage[all]{xy}
\usepackage{amsmath,amssymb,mathabx}
\usepackage{enumerate}
\usepackage{amsfonts}
\usepackage{sseq}
\usepackage{multicol}

\newtheorem{thm}{Theorem}[section]
\newtheorem{lem}[thm]{Lemma}
\newtheorem{prop}[thm]{Proposition}
\newtheorem{cor}[thm]{Corollary}

\theoremstyle{definition}
\newtheorem{rmk}{Remark}
\newtheorem{defn}{Definition}

\allowdisplaybreaks[1]
\raggedcolumns

\providecommand{\coker}{\operatorname{coker}}
\providecommand{\Ext}{\operatorname{Ext}}
\providecommand{\holim}{\operatorname{holim}}

\providecommand{\deg}{\operatorname{deg}}

\providecommand{\ZLOC}{\mathbb Z_{(3)}}

\providecommand{\QTWO}{Q(2)_{(3)}}
\newcommand{\Z}{\mathbb Z}
\newcommand{\N}{\mathbb N}

\providecommand{\LMAX}{\ell^{0,m}}
\providecommand{\LMAXB}{\ell^{1,m}}
\providecommand{\HOT}{\textrm{lower order terms}}

\providecommand{\BMF}{\mathcal B_{TMF}}
\providecommand{\floor}[1]{\left\lfloor#1
\right\rfloor}
\providecommand{\ceil}[1]{\left\lceil#1
\right\rceil}

\providecommand{\psitwo}{\psi_{[2]}}

\begin{document}


\title{On the Bousfield-Kan spectral sequence
for $\QTWO$}

\author{Donald M.\ Larson\thanks{Penn State
Altoona, 3000 Ivyside Park, Altoona PA, 16601; 
\texttt{dml34@psu.edu}}}

\maketitle

\begin{abstract}
We compute the $E_2$-term
of the Bousfield-Kan spectral sequence
converging to the homotopy groups
of the semi-cosimplicial $E_{\infty}$
ring spectrum $\QTWO$.
This $3$-local
spectrum was constructed by M.\ Behrens 
using degree 2 isogenies of elliptic curves,
and its localization
with respect to the 
2nd Morava $K$-theory $K(2)$
is ``one half'' of the
$K(2)$-local sphere at the prime 3.
The computation in this paper uses
techniques developed in 
the author's previous work 
\cite{Larson:ANSS} on the Adams-Novikov spectral 
sequence for $\QTWO$, and provides
another gateway to the homotopy ring $\pi_*\QTWO$.
\end{abstract}

\section{Introduction}\label{Introduction}
 
In his work on the $K(2)$-local sphere,
Behrens (\cite{Beh:Mod,Beh:Bldg}) 
constructs a $p$-local $E_{\infty}$
ring spectrum $Q(N)_{(p)}$ for each prime $p$ and
positive integer $N$ not divisible by $p$.
For fixed $p$ and $N$, this spectrum
is the homotopy limit of 
a semi-cosimplicial diagram $Q(N)_{(p)}^{\bullet}$ 
of the form
\begin{equation}\label{bullet}
Q(N)_{(p)}^{\bullet}:\quad 
TMF_{(p)}\Rightarrow 
TMF_0(N)_{(p)}\vee TMF_{(p)}\Rrightarrow
TMF_0(N)_{(p)}
\end{equation}
where $TMF_{(p)}$ and $TMF_0(N)_{(p)}$ are the 
$p$-localizations
of topological modular forms
and its analog for
$\Gamma_0(N)\subset\operatorname{SL}_2(\Z)$,
respectively.  The arrows in \eqref{bullet} 
denote alternating sums of two (resp.\ three)
coface maps, each defined
in terms of degree $N$ isogenies
of elliptic curves.

In the case $p=3$ and $N=2$,
the $K(2)$-localization
of \eqref{bullet} is a
reinterpretation of the
$K(2)$-local sphere resolution constructed by
Goerss, Henn, Mahowald, and Rezk
\cite{GHMR}.  Moreover, Behrens
(\cite{Beh:Mod}, Theorem 2.0.1) has shown that
\begin{equation}\label{cofiberseq}
DL_{K(2)}\QTWO\xrightarrow{DL_{K(2)}\eta}
L_{K(2)}S
\xrightarrow{L_{K(2)}\eta}L_{K(2)}\QTWO
\end{equation}
is a cofiber sequence, where $L_{K(2)}$ is
Bousfield localization with respect to $K(2)$,
$\eta$ is the unit map of $\QTWO$,
and $D$ is the $K(2)$-local Spanier-Whitehead
duality functor.  Therefore, the $K(2)$-local sphere
decomposes in terms of $L_{K(2)}\QTWO$ and 
$DL_{K(2)}\QTWO$, and the homotopy groups of these
three spectra are intertwined.  
The analog of \eqref{cofiberseq}
for $p=5$ and $N=2$
is known to be a cofiber sequence,
and is conjectured to be so
for all primes $p$ and corresponding $N$ (\cite{Beh:Bldg}, Conjecture 1.6.1).
The case $p=2$ is addressed by Behrens and
K.\ Ormsby in \cite{Behrens:Ormsby}.

In this paper,
we leverage the cosimplicial structure of 
$\QTWO$ 
and compute the $E_2$-term of the Bousfield-Kan
spectral sequence (BKSS) converging
to $\pi_*\QTWO$. 
We shall denote the $E_2$-term
of this spectral sequence by 
$\,_{BK}E_2^{s,t}\QTWO$.  
Our computation gives explicit descriptions
of the elements in this $E_2$-term
up to an ambiguity in a torsion $\ZLOC$-module
which we denote $U^*\subset\,_{BK}E_2^{1,*}\QTWO$.
Throughout this paper, $\nu_3(x)$ will
denote the 3-adic valuation of a (3-local)
integer $x$.
The following
is our main theorem.
\begin{thm}\label{main}
The Bousfield-Kan $E_2$-term for $\QTWO$ is
given by
\[
\,_{BK}E_2^{0,t}\QTWO=
\begin{cases}
\ZLOC,& t=0,\\
(\pi_tTMF_{(3)})_T, & t\neq0,
\end{cases}
\]
$\,_{BK}E_2^{1,t}\QTWO=(\pi_tTMF_{(3)})_T\oplus M^1$ where
\[M^1=
\begin{cases}
{\displaystyle \bigoplus_{n\in\N}(\ZLOC\oplus\Z/(3)),}&
t=0,\\
{\displaystyle\bigoplus_{n\in\N}\Z/(3)},&
t=4m, m<0,\\
\Z/(3^{\nu_3(3m)})\oplus
{\displaystyle\left(\bigoplus_{n\in\N}\Z/(3)\right)}, &
t=4m, m>0,\\
{\displaystyle\bigoplus_{n\in\N}\Z/(3)}, & t=4m+2, m\leq0,\\
U^t\oplus
{\displaystyle\left(\bigoplus_{n\in\N}\Z/(3)\right)},& t=4m+2, m\geq1, m\equiv13\ (27),\\
\Z/(3^{\nu_3(6m+3)})\oplus
{\displaystyle\left(\bigoplus_{n\in\N}\Z/(3)\right)}, & t=4m+2, m\geq1, m\nequiv13\ (27),\\
0, &\text{otherwise},
\end{cases}
\]
\[
\,_{BK}E_2^{2,t}\QTWO=
\begin{cases}
{\displaystyle\bigoplus_{n\in\N}(\ZLOC\oplus\Z/(3))},
& t=0,\\
{\displaystyle\bigoplus_{n\in\N}(\Z/(3)\oplus\Z/(3^{\nu_3(3m)}))},
& t=4m,m\neq0,\\
{\displaystyle\bigoplus_{n\in\N}(\Z/(3)\oplus
\Z/(3^{\nu_3(6m+3)}))},
& t=4m+2,m\leq0,\\
{\displaystyle\left(\bigoplus_{n\in\N}(\Z/(3)\oplus
\Z/(3^{\nu_3(6m+3)}))\right)\bigg/\sim},
& t=4m+2,m\geq1,m\equiv13\ (27),\\
{\displaystyle\bigoplus_{n\in\N}(\Z/(3)\oplus
\Z/(3^{\nu_3(6m+3)}))},
& t=4m+2,m\geq1,m\nequiv13\ (27)\\
0, &\text{otherwise}
\end{cases}
\]
where $\sim$ denotes a single relation among
the generators, and $\,_{BK}E_2^{s,t}\QTWO=0$
for $s\geq3$.
\end{thm}

In Section \ref{outline} we list 
the results that comprise our proof
of Theorem \ref{main}.  In Section \ref{ANSS_SECTION}
we set up the BKSS for $\QTWO$, along with
several algebraic tools needed for pursuing
the $E_2$-term.  Sections \ref{theghmaps} and
\ref{connecting} are the most technical, and
contain the computations needed to finish the proof of
Theorem \ref{main}.  Finally, in Section
\ref{higherdiffs} we briefly examine the structure of 
the differentials on the $E_2$-page and
beyond in the BKSS.

\section{Statement of main results}
\label{outline}
This section outlines our strategy for
proving Theorem \ref{main}.  Our approach
resembles that from our previous work
\cite{Larson:ANSS} on the Adams-Novikov
$E_2$-term for $\QTWO$.  In Propositions
\ref{BKSSETWO} -- \ref{gandh} below, there
are references to two graded $\ZLOC$-algebras,
$B$ and $\Gamma$; these algebras piece together
to form an elliptic curve Hopf algebroid
that will be defined in Section \ref{ANSS_SECTION}.
This Hopf algebroid is the key algebraic
object underlying our computation.

To begin,
we express the Bousfield-Kan $E_2$-term
for $\QTWO$ algebraically.

\begin{prop}\label{BKSSETWO}
The BKSS $E_2$-term for $\QTWO$ is
the cohomology of a semi-cosimplicial
abelian group.  More precisely,
it has the form
\[
\,_{BK}E_2^{s,t}\QTWO=H^s(
\pi_tTMF_{(3)}\Rightarrow B_{t/2}
\times\pi_tTMF_{(3)}\Rrightarrow
B_{t/2}
)
\]
where the coface maps are induced
by the corresponding maps in \eqref{bullet}.
\end{prop}
For notational convenience,
we put
\begin{equation}\label{gbullet}
\mathcal G_*^{\bullet}:=
(\pi_*TMF_{(3)}\Rightarrow B_{*/2}
\times\pi_*TMF_{(3)}\Rrightarrow
B_{*/2})
\end{equation}
so that $\,_{BK}E_2^{s,t}\QTWO=H^s\mathcal G_t^{\bullet}$
by Proposition \ref{BKSSETWO}.
The following proposition describes
a two-stage filtration that we use to compute 
$H^*\mathcal G_*^{\bullet}$.

\begin{prop}\label{filtration}
There is a filtration 
$\mathcal G_*^{\bullet}=F^0
\supset F^1\supset F^2$ of 
$\mathcal G_*^{\bullet}$
inducing a short exact sequence
$0\to C'\to \mathcal G_*^{\bullet}\to C''\to 0$,
where 
\[
C'=(0\to B_{*/2}
\xrightarrow{h}
B_{*/2}),\quad
C''=(\pi_*TMF_{(3)}\xrightarrow{g}\pi_*TMF_{(3)}\to0).
\]
The resulting long exact sequence in
cohomology is
\[
0\to H^0\mathcal G_*^{\bullet}\to\ker g
\xrightarrow{\delta^0}
\ker h\to H^1\mathcal G_*^{\bullet}\to\coker g
\xrightarrow{\delta^1}
\coker h\to H^2\mathcal G_*^{\bullet}\to0
\]
so that $H^0\mathcal G_*^{\bullet}=\ker\delta^0$,
$H^2\mathcal G_*^{\bullet}=\coker\delta^1$,
and $H^1\mathcal G_*^{\bullet}$ lies in the
short exact sequence
\begin{equation}\label{hone}
0\to\coker\delta^0\to H^1\mathcal G_*^{\bullet}
\to\ker\delta^1\to0.
\end{equation}
Moreover, 
$\coker\delta^0$
is concentrated in $t$-degree 0
and $\ker\delta^1$ is free of rank 1
in $t$-degree 0, so that the sequence
\eqref{hone} splits in $t$-degree 0
and $H^1\mathcal G_*^{\bullet}\cong\ker\delta^1$
in nonzero $t$-degrees.
\end{prop}
Leveraging the filtration given in 
Proposition \ref{filtration} requires
computing the kernels and cokernels
of the maps $g$ and $h$.  We do so
by applying some technical results from
our previous work \cite{Larson:ANSS}
on the ANSS for $\QTWO$.  
Throughout this paper, $A_T$ will denote
the torsion subgroup of an abelian group $A$.
\begin{prop}\label{gandh}
As $\ZLOC$-modules,
\begin{align*}
\ker g&=\pi_0TMF_{(3)}\oplus
(\pi_*TMF_{(3)})_T,\\
\ker h&=\bigoplus_{n\in\mathbb N}
\ZLOC,
\end{align*}
and
\begin{align*}
\coker g&=\pi_*TMF_{(3)}/\sim_{\coker g},\\
\coker h&=
\left(\bigoplus_{n\in\N}\ZLOC\right)\oplus
\left(\bigoplus_{\substack{i<j\in\Z\\i+j\neq0}}
\mathbb Z/(3^{\nu_3(i+j)+1})
\right)\oplus\left(
\bigoplus_{i<j\in\Z}
\mathbb Z/(3^{\nu_3(2i+2j+1)+1})
\right)
\end{align*}
where $\sim_{\coker g}$ denotes the set of all
relations of the form $3^{\nu_3(t)+1}\cdot x=0$
for $x\in\pi_*TMF_{(3)}$ represented by an element 
in $\Ext^{0,t}_{\Gamma}
(B,B)$ with $t\neq0$.
\end{prop}

The following theorem describes 
$H^*\mathcal G_*^{\bullet}$.  
Using 
results from \cite{Larson:ANSS} once again,
we
prove this result by computing the connecting homomorphisms
$\delta^0$ and $\delta^1$ from Proposition \ref{filtration}.
\begin{thm}\label{cohomology}
\begin{enumerate}[(a)]
\item $H^0\mathcal G_*^{\bullet}=
\ZLOC\oplus(\pi_*TMF_{(3)})_T$
\item 
${\displaystyle
H^1\mathcal G_*^{\bullet}=\left(\bigoplus_{n\in\N}\ZLOC
\right)\oplus\left(\bigoplus_{n\in\N}\Z/(3)\right)
\oplus(\pi_*TMF_{(3)})_T\oplus\left(
\bigoplus_{m>0}\Z/(3^{\nu_3(3m)})
\right)
}$ \\
${\displaystyle\qquad\qquad\qquad\qquad
\quad\oplus\left(
\bigoplus_{\substack{m>0\\m\nequiv13\ (27)}}
\Z/(3^{\nu_3(6m+3)})
\right)\oplus U^*}$
\item 
${\displaystyle H^2\mathcal G_*^{\bullet}=
\left(\bigoplus_{n\in\N}\ZLOC\right)
\oplus\left(\bigoplus_{n\in\N}\Z/(3)\right)
\oplus\left(
\bigoplus_{m\neq0}\bigoplus_{n\in\N}\Z/(3^{\nu_3(3m)})
\right)\oplus
\left(
\bigoplus_{m\leq0}\bigoplus_{n\in\N}\Z/(3^{\nu_3(6m+3)})
\right)
}$\\
${\displaystyle\qquad\quad\oplus\left(
\bigoplus_{\substack{m>0\\m\nequiv13\ (27)}}
\bigoplus_{n\in\N}\Z/(3^{\nu_3(6m+3)})
\right)\oplus\left(
\bigoplus_{\substack{m>0\\m\equiv13\ (27)}}\left(\bigoplus_{n\in\N}
(\Z/(3)\oplus\Z/(3^{\nu_3(6m+3)}))\right)\bigg/\sim
\right)
}$
\end{enumerate}
\end{thm}
\begin{proof}[Proof of Theorem \ref{main}.]\let\qed\relax
Our proof of Theorem \ref{cohomology}
will reveal that the generators of $\ZLOC$
summands lie in $\,_{BK}E_2^{*,0}\QTWO$,
the generators of $\Z/(3^{\nu_3(3m)})$ 
(resp.\ $\Z/(3^{\nu_3(6m+3)})$)
lie in $\,_{BK}E_2^{*,4m}\QTWO$
(resp.\ $\,_{BK}E_2^{*,4m+2}\QTWO$),
and the direct sums of countably many copies of $\Z/(3)$
lie in the bidegrees
Theorem \ref{main} indicates.  
The relation $\sim$ appearing in Theorem \ref{main}
and Theorem \ref{cohomology}(c) is defined
below in Lemma \ref{LEMMAFIVE}(a).
Coupled with
these facts, Propositions \ref{BKSSETWO} and
\ref{filtration} and Theorem \ref{cohomology}
together prove Theorem \ref{main}.
\end{proof}

\section{$\pmb{\QTWO}$ and its Bousfield-Kan
spectral sequence}
\label{ANSS_SECTION}
In this section we set up the Bousfield-Kan
spectral sequence for $\QTWO$ and its underlying
algebraic apparatus.  In particular,
we prove Propositions \ref{BKSSETWO}
and \ref{filtration}.

\subsection{The homotopy of $\pmb{TMF_{(3)}}$
and $\pmb{TMF_0(2)_{(3)}}$}
We begin by formally defining $\QTWO$
in terms of the semi-cosimplicial diagram \eqref{bullet}
with $p=3$, $N=2$.  The coface maps will be defined
later in this section in terms of algebraic data.
\begin{defn}\label{cosimplicial} The spectrum
$\QTWO$ is given by
\[
\QTWO:=\holim(TMF_{(3)}\Rightarrow TMF_0(2)_{(3)}\vee 
TMF_{(3)}
\Rrightarrow TMF_0(2)_{(3)}),
\]
that is, $\QTWO=\holim\QTWO^{\bullet}$.
\end{defn}
The homotopy rings $\pi_*TMF_{(3)}$
and $\pi_*TMF_0(2)_{(3)}$ will play a key
role in our computation, 
and they both arise from the data in
an elliptic curve Hopf algebroid 
over $\ZLOC$ that we shall denote
$(B_*,\Gamma_*)$ (we will often
suppress the grading and simply write $(B,\Gamma)$).  
This Hopf
algebroid co-represents the groupoid
of non-singular elliptic curves over 
$\ZLOC$ with Weierstrass equation
\begin{equation}\label{Weier}
y^2=4x(x^2+q_2x+q_4)
\end{equation}
and isomorphisms $x\mapsto x+r$ that preserve
this Weierstrass form. 
\begin{defn}
As graded $\ZLOC$-algebras,
\begin{align*}
B_*&:=\ZLOC[q_2,q_4,\Delta^{-1}]/
(\Delta=q_4^2(16q_2^2-64q_4)),\\
\Gamma_*&:=B_*[r]/(r^3+q_2r^2+q_4r)
\end{align*}
where $q_2\in B_2$, $q_4\in B_4$
(hence $\Delta\in B_{12}$), and
$r\in\Gamma_2$.
\end{defn} 
The following lemma gives Adams-Novikov-style
spectral sequences converging to $\pi_*TMF_{(3)}$
and $\pi_*TMF_0(2)_{(3)}$, respectively
(see, e.g., \cite{Beh:Mod}, Corollary 1.4.2).
\begin{lem}\label{ANSSFORTMF}
\begin{enumerate}[(a)]
\item The Adams-Novikov spectral sequence (ANSS) converging
to $\pi_*TMF_{(3)}$ has the form
\begin{equation}\label{ANSSETWO}
\,_{AN}E_2^{s,t}TMF_{(3)}:=\Ext^{s,t}_{\Gamma}(B,B)
\Rightarrow\pi_{2t-s}TMF_{(3)}
\end{equation}
where $\Ext^{*,*}_{\Gamma}(B,B)$ is the 
cohomology of $(B,\Gamma)$.
\item The ANSS converging
to $\pi_*TMF_0(2)_{(3)}$ has the form
\[
\,_{AN}E_2^{s,t}TMF_0(2)_{(3)}:=
\Ext_B^{s,t}(B,B)\Rightarrow\pi_{2t-s}
TMF_0(2)_{(3)}
\]
and therefore collapses at the $E_2$-page, yielding 
\[
\pi_{2k}TMF_0(2)_{(3)}=B_k 
\]
for all $k\in\Z$.  In particular, 
$\pi_*TMF_0(2)_{(3)}$ is concentrated
in dimensions congruent to 0 modulo 4.
\end{enumerate}
\end{lem}
The spectral sequence \eqref{ANSSETWO} has
been computed by Hopkins and Miller.  Expository
accounts of this computation can be found in 
\cite{Bau}, \cite{Henriques}, and \cite{Mathew}.
We recall the results of this computation here.
\begin{lem}\label{TMFHOMOTOPY}
The homotopy ring of $TMF_{(3)}$ is
\[
\pi_*TMF_{(3)}=\ZLOC[c_4,c_6,3\Delta,3\Delta^2,
\Delta^3,\Delta^{-3},\alpha,\beta,b]/\sim_{TMF}
\]
where 
$c_4\in\pi_8$, $c_6\in\pi_{12}$,
$3\Delta\in\pi_{24}$, $\alpha\in\pi_3$,
$\beta\in\pi_{10}$, $b\in\pi_{27}$,
and 
\[
\sim_{TMF}\, =\begin{cases}
c_4^3-c_6^2=576\cdot3\Delta,\\
3\alpha=3\beta=3b=0,\\
\alpha\cdot3\Delta=\alpha\cdot3\Delta^2
=\beta\cdot3\Delta=\beta\cdot3\Delta^2=0,\\
\alpha^2=\alpha\beta^2=\beta^5=0,\\
c_4\alpha=c_6\alpha=
c_4\beta=c_6\beta=
c_4b=c_6b=0.
\end{cases}
\]
\end{lem}
\begin{figure}[h!]
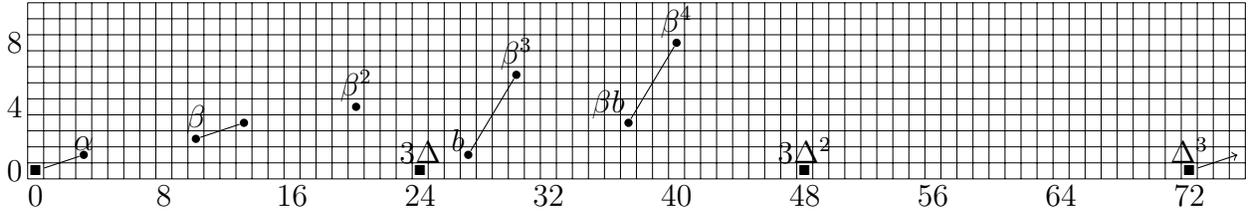

\centering
\begin{sseq}[entrysize=2.13mm,ylabelstep=4,xlabelstep=8]{0...75}{0...10}
\ssdrop{\sqbullet}
\ssmove 3 1
\ssdropbull
\ssdroplabel[U]{\alpha}
\ssstroke
\ssmoveto{10}{2} 
\ssdropbull
\ssdroplabel[U]{\beta}
\ssmove 3 1
\ssdropbull
\ssstroke
\ssmoveto{20}{4}
\ssdropbull
\ssdroplabel[U]{\beta^2}
\ssmoveto{24}0
\ssdrop{\sqbullet}
\ssdroplabel[U]{3\Delta}
\ssmoveto{27}{1}
\ssdropbull
\ssdroplabel[LU]{b}
\ssmoveto{30}6
\ssdropbull
\ssdroplabel[U]{\beta^3}
\ssstroke
\ssmoveto{37}3
\ssdropbull
\ssdroplabel[LU]{\beta b}
\ssmoveto{40}8
\ssdropbull
\ssstroke
\ssdroplabel[U]{\beta^4}
\ssmoveto{48}0
\ssdrop{\sqbullet}
\ssdroplabel[U]{3\Delta^2}
\ssmoveto{72}0
\ssdrop{\sqbullet}
\ssdroplabel[U]{\Delta^3}
\ssvoidarrow 3 1
\end{sseq}
\caption{The $E_{\infty}$-term of
the ANSS for $TMF_{(3)}$, $0\leq 2t-s\leq75$.}
\label{EINFINITY}
\end{figure}

The $E_{\infty}$-term
of \eqref{ANSSETWO} is shown 
(with some elements omitted)
in Figure \ref{EINFINITY}, with $s$ along
the vertical axis and $2t-s$ 
along the horizontal axis.  
A black square denotes a copy of $\ZLOC$
and a black circle denotes a copy of $\Z/(3)$.
Lines
of slope $1/3$ represent multiplication
by $\alpha$, and lines of slope 
$5/3$ represent multiplicative
extensions induced by the relation
$b\alpha=\beta^3$ in $\pi_{27}$.  
Not shown in the chart are most of the elements
on the line $s=0$; these elements will
be enumerated below in Subsection \ref{algebra}.
The portion of the $E_{\infty}$-term shown
is 72-periodic 
with invertible generator
$\Delta^3$.
\begin{rmk}\label{cobar}
The groups $\Ext^{*,*}_{\Gamma}(B,B)$
are encoded as the cohomology groups of the
{\em cobar resolution} for $(B,\Gamma)$,
a cochain complex of the form
\[
B\to\Gamma\to\Gamma\otimes\Gamma\to
\Gamma\otimes\Gamma\otimes\Gamma\to\cdots
\]
(\cite{Rav:MU}, A1.2.11)
where the differentials are defined in terms
of the structure maps of $(B,\Gamma)$.
For example, $\alpha\in\pi_3TMF_{(3)}$
is represented by $r\in\Gamma$ in the cobar
resolution; similarly,
$b\in\pi_{27}TMF_{(3)}$ is represented
by $r\Delta\in\Gamma$, and
$\beta\in\pi_{10}TMF_{(3)}$ is represented
by $r^2\otimes r-r\otimes r^2\in\Gamma\otimes\Gamma$.
Representatives for the other
torsion elements in $\pi_*TMF_{(3)}$ can be computed
via the cobar resolution pairing
defined in (\cite{Rav:MU}, A1.2.15).
\end{rmk}

\subsection{Coface maps and the BKSS setup}

To define the coface maps appearing
in Definition \ref{cosimplicial},
we will need four Hopf algebroid
maps:
\begin{align*}
\psi_d:(B,B)&\to(B,B),\\
\phi_f:(B,\Gamma)&\to(B,B),\\
\phi_q:(B,\Gamma)&\to(B,B),\\
\psitwo:(B,\Gamma)&\to(B,\Gamma).
\end{align*}
\begin{defn}\label{themaps}
The maps $\psi_d$, $\phi_f$,
$\phi_q$, and $\psitwo$
are induced by the following
maps of $\ZLOC$-algebras:
\begin{multicols}{4}
\noindent\begin{align*}
\psi_d:B&\to B\\
q_2&\mapsto-2q_2,\\
q_4&\mapsto q_2^2-4q_4.
\end{align*}
\begin{align*}
\phi_f:\Gamma&\to B\\
q_2&\mapsto q_2,\\
q_4&\mapsto q_4,\\
r&\mapsto 0.
\end{align*}
\begin{align*}
\phi_q:\Gamma&\to B\\
q_2&\mapsto -2q_2,\\
q_4&\mapsto q_2^2-4q_4,\\
r&\mapsto 0.
\end{align*}
\begin{align*}
\psitwo:\Gamma&\to\Gamma\\
q_2&\mapsto 4q_2,\\
q_4&\mapsto 16q_4,\\
r&\mapsto 4r.
\end{align*}
\end{multicols}
\end{defn}
\begin{rmk}
These maps 
are defined in terms of certain maneuvers with
Weierstrass equations of elliptic curves,
and are therefore most naturally defined
in the context of algebraic stacks. 
We shall not discuss these maneuvers here, 
but the details can be found in 
(\cite{Beh:Mod}, Section 1.5) and are recapitulated
in (\cite{Larson:ANSS}, Section 3.3).
\end{rmk}
We may now define the coface maps.  We do so
on the level of Hopf algebroids, and the maps
in Definition \ref{cosimplicial} are the corresponding
maps of spectra (which, by abuse
of notation, will be denoted
using the same symbols).
\begin{defn}
The map
out of $TMF_{(3)}$ is induced by
the alternating sum
$d_0-d_1$, where
\[
d_0:=\phi_q\oplus\psitwo,\quad
d_1:=\phi_f\oplus 1_{(B,\Gamma)},
\]
and the map into $TMF_0(2)_{(3)}$
is induced by the alternating sum
$\overline{d_0}-\overline{d_1}+
\overline{d_2}$, where
\[
\overline{d_0}:=\psi_d,\quad
\overline{d_1}:=\phi_f,\quad
\overline{d_2}:=1_{(B,B)}.
\]
\end{defn}

To elucidate the structure of the BKSS
for $\QTWO$, we start with an observation
of Behrens \cite{Beh:Cong} about the
$E_1$-term, namely that
\[
\,_{BK}E_1^{s,t}\QTWO=\pi_t\QTWO^s\Rightarrow
\pi_{t-s}\QTWO.
\]
In other words, the $E_1$-term is 
the semi-cosimplicial abelian group
$\pi_*\QTWO^{\bullet}=\mathcal G_*^{\bullet}$.
The $E_2$-term
is therefore obtained by taking
cohomology of 
$\mathcal G_*^{\bullet}$ at each of its
three nontrivial stages.
Within the general framework of the Bousfield-Kan
technology, 
the relevant notion is that of cohomotopy.
\begin{defn}\label{cohomotopy}
The {\em $s$th cohomotopy group} 
$\pi^sG^{\bullet}$
of a a (semi-)cosimplicial abelian group
$G^{\bullet}$ is given by
\[
\pi^sG^{\bullet}:=H^sG^{\bullet}
\]
where cohomology is taken with respect to
the alternating sums of the coface maps
of $G^{\bullet}$.
\end{defn}
Therefore,  
\[
\,_{BK}E_2^{s,t}\QTWO=\pi^s\pi_t\QTWO^{\bullet}
=H^s\pi_t\QTWO^{\bullet}
\]
and since $\pi_tTMF_0(2)_{(3)}=B_{t/2}$
by Lemma \ref{ANSSFORTMF}(b),
this proves Proposition \ref{BKSSETWO}.

\subsection{Filtration of $\pmb{\mathcal G_*^{\bullet}}$}
We now define the filtration $\mathcal G_*^{\bullet}
=F^0\supset F^1\supset F^2$ asserted to exist in Proposition
\ref{filtration}.  For ease of notation, we will henceforth
denote by 1 the maps $1_{(B,B)}$, $1_{(B,\Gamma)}$,
and any maps they induce or correspond to; 
the meaning will be clear
from the context.

The filtration we want is $F^1=(\pi_*TMF_{(3)}
\xrightarrow{\psitwo-1}\pi_*TMF_{(3)}\to0)$
and $F^2$ is the trivial complex.
The induced short exact sequence is given by
\[
\xymatrix{
C': & 0\ar[r]\ar[d] 
    & B_{\,}\ar[r]^{\psi_d+1}\ar@{^{(}->}[d] 
    & B\ar@{=}[d] \\
\mathcal G_*^{\bullet}:  & \pi_*TMF_{(3)}\ar[r]\ar@{=}[d] 
    & B\oplus \pi_*TMF_{(3)}\ar[r]\ar@{>>}[d] 
    & B\ar[d]\\
C'':& \pi_*TMF_{(3)}\ar[r]^{\psi_{[2]}-1} & 
\pi_*TMF_{(3)}\ar[r] & 0
}
\]
which leads to the following definition
of the maps $g$ and $h$ appearing in the statement
of Proposition \ref{filtration}.
\begin{defn}\label{defgh}
$g:=\psitwo-1:\pi_*TMF_{(3)}\to\pi_*TMF_{(3)}$, 
$h:=\psi_d+1:B\to B$.
\end{defn}
With the exception of the splitting
of the short exact sequence \eqref{hone},
Proposition \ref{filtration} follows from
standard homological algebra (see, e.g.,
Section 1.3 of \cite{Weibel}).  The connecting 
map $\delta^0$ is the restriction of $\phi_q-\phi_f$
to $\ker g$, while the connecting map $\delta^1$
is induced by $-\phi_f$.

\subsection{Algebraic properties of $\pmb{B_*}$ and 
$\pmb{\pi_*TMF_{(3)}}$}\label{algebra}

In this subsection we
lay the algebraic groundwork
for our computation by examining
the rings $B_*$ and $\pi_*TMF_{(3)}$.
The results and notation from this
subsection parallel those in 
(\cite{Larson:ANSS}, Subsection 3.2).

We start with $B_*$, where
we define a new elements $\sigma,\tau\in B_4$
as follows (cf.\ \cite{Larson:ANSS}, Eq.\ (10)):
\[ 
\sigma:=8q_4,\quad
\tau:=\frac{16q_2^2-64q_4}8.
\]
\begin{defn}\label{evectors}
For $i<j\in\mathbb Z$, 
$a_{i,j}:=\sigma^i\tau^j-\sigma^j\tau^i$ and
$b_{i,j}:=a_{i,j}q_2$.
\end{defn}
The elements $\{a_{i,j}\}$ and $\{b_{i,j}\}$
from Definition \ref{evectors} are collectively
a subset of a basis of eigenvectors for $B$ with respect
to the map $h$; this was proven in (\cite{Larson:ANSS},
Subsection 4.3).
The following definition gives a convenient
enumeration of these elements, particularly for our
study of the connecting map
$\delta^1$ in Subsection \ref{deltaonethree}.
\begin{defn}\label{defnAB} For $0\leq v\in\Z$ and $m\in\Z$,
\[
A_v^m:=a_{\floor{\frac{m-1}2}-v,\ceil{\frac{m+1}2}+v},\quad
B_v^m:=b_{\floor{\frac{m-1}2}-v,\ceil{\frac{m+1}2}+v}.
\]
\end{defn}
The following lemma is a restatement
of (\cite{Larson:ANSS}, Lemma 3).
\begin{lem}\label{bdirectsum}
${\displaystyle
B=\bigoplus_{v\geq0,m\in\Z}\ZLOC\{A_v^m,B_v^m\}
}$
\end{lem}
Next we record some results on $\pi_*TMF_{(3)}$.
Specifically, we will focus on elements
in $\,_{AN}E_{\infty}^{0,*}TMF_{(3)}$; that is,
elements on the 0-line of the ANSS for $TMF_{(3)}$.
\begin{lem}\label{TMFBASIS}
If
\[
I_{TMF}:=\{(n,\epsilon,\ell_1,
\ell_2,\ell_3):0\leq n\in\Z,\epsilon\in
\{0,1\},(\ell_1,\ell_2)\in
\{(0,0),(0,1),(1,0)\},\ell_3\in\Z
\},
\]
then
\[
\ _{AN}E_{\infty}^{0,*}TMF_{(3)}
=\ZLOC\{c_4^n\cdot c_6^{\epsilon}\cdot
[3\Delta]^{\ell_1}\cdot
[3\Delta^2]^{\ell_2}\cdot
[\Delta^3]^{\ell_3}:(n,\epsilon,
\ell_1,\ell_2,\ell_3)\in I_{TMF}\}.
\]
\end{lem}
Lemma \ref{TMFBASIS} gives a basis for 
$\ _{AN}E_{\infty}^{0,*}TMF_{(3)}$ over $\ZLOC$.
The following lemma gives representatives
for these basis elements in $\Ext_{\Gamma}^{0,*}(B,B)$
so that we may compute the maps into/out of 
$\pi_*TMF_{(3)}$ (or subquotients thereof)
in the BKSS.  
\begin{lem}\label{basisTMF}
\begin{enumerate}[(a)]
\item $\Ext_{\Gamma}^{0,*}(B,B)
=\ZLOC[c_4,c_6,\Delta,\Delta^{-1}]/(1728\Delta=c_4^3-c_6^2)$
\item If
\[
\mathcal B^{\epsilon,m}_{TMF}
:=\{
\gamma c_4^nc_6^{\epsilon}\Delta^{\ell}:
n\geq0,\epsilon\in\{0,1\},\ell\in\Z,
n+\epsilon+3\ell=m
\}\subset\Ext_{\Gamma}^{0,*}(B,B)
\]
where 
\[\gamma:=
\begin{cases}
1, & \text{if }\ell\equiv 0\, (3),\\
3, & \text{otherwise,}
\end{cases}
\]
then
\[
\mathcal B_{TMF}:=\coprod_{m\in\Z, \epsilon\in\{0,1\}}
\mathcal B^{\epsilon,m}_{TMF}
\]
is a complete set of representatives for the 
$\ZLOC$-basis of $\ _{AN}E_{\infty}^{0,*}TMF_{(3)}$
given in Lemma \ref{TMFBASIS}.
\end{enumerate}
\end{lem}
\begin{proof}
Part (a) is proven in \cite{Del:Ell}.
Part (b) follows from Lemma \ref{TMFHOMOTOPY}.
\end{proof}
As we do throughout this paper, we shall
not distinguish between elements of 
$\mathcal B_{TMF}$ and the
homotopy elements they represent
in $\ _{AN}E_{\infty}^{0,*}TMF_{(3)}
\subset\pi_*TMF_{(3)}$.
In particular, we may interpret 
Lemma \ref{basisTMF} as saying that 
$\mathcal B_{TMF}$
{\em is} a $\ZLOC$-basis for
$\,_{AN}E_{\infty}^{0,*}TMF_{(3)}$.

We now give notation for the submodules 
of $\pi_*TMF_{(3)}$ spanned by the sets 
$\mathcal B^{\epsilon,m}_{MF}$ defined
in Lemma \ref{basisTMF}(b).
\begin{defn}\label{TMFSUBMODULES}
Given $\epsilon\in\{0,1\}$ and $m\in\Z$, 
\[
W^{\epsilon,m}:=\ZLOC\{
\mathcal B^{\epsilon,m}_{TMF}\}
\subset \,_{AN}E_{\infty}^{0,*}TMF_{(3)}.
\]
\end{defn}
\begin{lem}\label{zerolinedecomp}
As a $\ZLOC$-module,
${\displaystyle
\,_{AN}E_{\infty}^{0,*}TMF_{(3)}=
\bigoplus_{m\in\Z,\epsilon\in\{0,1\}} W^{\epsilon,m}}$.
\end{lem}
\begin{proof}
This follows from Definition \ref{TMFSUBMODULES}
and the disjoint union decomposition of 
$\mathcal B_{TMF}$ given in Lemma \ref{basisTMF}(b).
\end{proof}
For a monomial 
$\gamma c_4^{n}
c_6^{\epsilon}
\Delta^{\ell}
\in\mathcal B_{TMF}^{\epsilon,m}$, let
$\ell^{\epsilon,m}$ be
the largest possible value of 
$\ell$.  Then
\[
\ell^{0,m}=\floor{\dfrac m3},\quad\ell^{1,m}=
\floor{\dfrac{m-1}3}.
\]
Using these quantities, we can 
enumerate the elements in $\mathcal B_{TMF}$ in
a way that is convenient for our
study of $\delta^1$ in Subsection \ref{deltaonethree}.
\begin{defn}\label{defnCD}
For $0\leq v\in\Z$ and $m\in\Z$,
\[
C_v^m:=c_4^{m-3\ell^{0,m}+3v}\Delta^{\ell^{0,m}-v},
\quad
D_v^m:=c_4^{m-3\ell^{1,m}+3v-1}c_6\Delta^{\ell^{1,m}-v}.
\]
\end{defn}
\begin{prop}\label{TMFBASES}
For $m\in\Z$,
\[\mathcal B_{TMF}^{0,m}
=\{\gamma_0C_0^m,
\gamma_1C_1^m,
\gamma_2C_2^m,\ldots\},\qquad\mathcal B_{TMF}^{1,m}=\{\theta_0D_0^m,\theta_1D_1^m,
\theta_2D_2^m,\ldots\}\]
where
\[
\{\gamma_0,\gamma_1,\gamma_2,\ldots\}:=
\begin{cases}
\{1,3,3,1,3,3,\ldots\}, & m\equiv0,1,2\ (9),\\
\{3,1,3,3,1,3,3,\ldots\}, & m\equiv3,4,5\ (9),\\
\{3,3,1,3,3,1,\ldots\},&\text{otherwise}
\end{cases}
\]
and
\[
\{\theta_0,\theta_1,\theta_2,\ldots\}:=
\begin{cases}
\{1,3,3,1,3,3,\ldots\}, & m\equiv1,2,3\ (9),\\
\{3,1,3,3,1,3,3,\ldots\}, & m\equiv4,5,6\ (9),\\
\{3,3,1,3,3,1,\ldots\},&\text{otherwise}.
\end{cases}
\]
\end{prop}
\begin{proof}
Only the definitions of the sequences
$\{\gamma_0,\gamma_1,\gamma_2,\ldots\}$
and $\{\theta_0,\theta_1,\theta_2,\ldots\}$
require justification.  By Lemma \ref{basisTMF}(b)
and Definition \ref{defnCD}, the value of $\gamma_v$
in $\gamma_vC_v^m$ depends on whether $\LMAX-v$
is divisible by 3.  Since $\LMAX=\floor{m/3}$,
this 3-divisiblity in turn depends on the value
of $m$ modulo 9, as well as the value of $v$.  
An elementary calculation then shows that the above
definition of 
the sequence $\{\gamma_v\}$ is correct.  
An analogous argument justifies the above definition
of the sequence $\{\theta_v\}$.
\end{proof}
\begin{cor}\label{TMFZERO}
$\pi_0TMF_{(3)}=\ZLOC
\{C_0^0=1,3C_1^0,3C_2^0,
C_3^0,3C_4^0,3C_5^0,
C_6^0,3C_7^0,3C_8^0,\ldots\}
$.
\end{cor}
\begin{proof}
This follows from Proposition \ref{TMFBASES}
and the fact that $\pi_0TMF_{(3)}=\,_{AN}E_{\infty}^{0,0}
TMF_{(3)}$.
\end{proof}
\begin{rmk}\label{degreecounting}
The enumerations in Definitions \ref{defnAB}
and \ref{defnCD} are analogous in terms of how 
the integer $m$ compares with the internal degree $t$.  Specifically,
\begin{equation}
\deg_t(A_v^m)=\deg_t(\gamma_vC^m_v)
=4m,\quad
\deg_t(B_v^m)=\deg_t(\theta_vD^m_v)
=4m+2.
\end{equation}
\end{rmk}

\section{Computation of the maps $\pmb{g}$ and $\pmb{h}$}
\label{theghmaps}
In this section we initiate our computation
of $\,_{BK}E_2^{*,*}\QTWO$ by computing
the kernel and cokernel of the maps $g:\pi_*TMF_{(3)}
\to\pi_*TMF_{(3)}$
and $h:B\to B$ from Definition \ref{defgh}.
These computations prove Proposition \ref{gandh}.

To begin, we record a useful result in 3-adic
analysis.
\begin{lem}[\cite{Larson:ANSS}, Lemma 8(a)]\label{adic}
If $0\neq n\in\Z$, then 
$\nu_3(4^n-1)=\nu_3(n)+1$.
\end{lem}
Next, we compute the effect of $g$ on
an element of $\pi_*TMF_{(3)}$ (cf. \cite{Larson:ANSS},
Eq.\ (21)).
\begin{lem}\label{geffect}
If $x\in\pi_{2t-s}TMF_{(3)}$ is represented
by an element of $\Ext_{\Gamma}^{s,t}(B,B)$
in the spectral sequence \eqref{ANSSETWO},
then
\[
g(x)=(2^t-1)x.
\]
\end{lem}
\begin{proof}
Suppose the representative of $x$ 
in $\Ext_{\Gamma}^{s,t}(B,B)$ 
is itself represented in the cobar
resolution (see Remark \ref{cobar})
by $x'\in(\Gamma^{\otimes s})_t$.
Since $g=\psitwo-1$, the formulas
from Definition \ref{themaps} 
imply that
\[
g:x'\mapsto(2^t-1)x'.
\]
\end{proof}
\begin{cor}\label{gcor}
$\ker g=\pi_0TMF_{(3)}\oplus
(\pi_*TMF_{(3)})_T$, and
\[
\coker g=\pi_*TMF_{(3)}/\sim_{\coker g}
\]
where $\sim_{\coker g}$ denotes the set of all
relations of the form $3^{\nu_3(t)+1}\cdot x=0$
for $x$ represented by an element in $\Ext^{0,t}_{\Gamma}
(B,B)$ with $t\neq0$.
\end{cor}
\begin{proof}
If $t=0$, then $2^t-1=0$.
By degree counting, the elements
in the ANSS for $TMF_{(3)}$ with $t$-degree 0
are exactly those with topological degree
$2t-s=0$.  Thus $\pi_0TMF_{(3)}\subset\ker g$
by Lemma \ref{geffect}.

Elements in $(\pi_*TMF_{(3)})_T$ have
$t$-degree nonzero and even, 
and they
generate copies of $\Z/(3)$.
Since $2^t-1=4^{t/2}-1$ is always divisible
by 3 for $t$ even,
Lemma \ref{geffect} implies $(\pi_*TMF_{(3)})_T\subset\ker g$.

The remaining elements in $\pi_*TMF_{(3)}$, namely the nontrivial
elements in $\,_{AN}E_{\infty}^{0,t}\QTWO$ for $t\neq0$, 
are non-torsion, and therefore are not in $\ker g$.
The formula for $\ker g$ follows.

The formula for $\coker g$ follows immediately
from Lemmas \ref{adic} and \ref{geffect}.
\end{proof}
\begin{rmk}\label{therelation}
The relation $\sim_{\coker g}$ has no effect 
on elements of $(\pi_*TMF_{(3)})_T$ since
they all generate copies of $\Z/(3)$; it also
has no effect on elements of $\pi_0TMF_{(3)}$
since they are precisely the elements whose
representatives in the ANSS lie in $t$-degree 0.
The affected elements are precisely the remaining
ones, namely those
whose representatives lie on the 0-line $\,_{AN}E_{\infty}^{0,t}
TMF_{(3)}$ with $t\neq0$.
\end{rmk}

The kernel and cokernel of $h$ were computed
by the author in \cite{Larson:ANSS}.
\begin{prop}[\cite{Larson:ANSS},
Proposition 3]\label{hprop}
\begin{enumerate}[(a)]
\item $\ker h=\Z_3\{a_{-i,i}:i\geq1\}$
\item ${\displaystyle\coker h=\bigoplus_{i<j\in\Z}\left(
\Z/3^{\nu_3(i+j)+1}\{a_{i,j}\}\oplus
\Z/3^{\nu_3(2i+2j+1)+1}\{b_{i,j}\}
\right)}$
\end{enumerate}
\end{prop}
Corollary \ref{gcor} and Proposition \ref{hprop}
together imply Proposition \ref{gandh}.

\section{Computation of $\pmb{\delta^0}$ and $\pmb{\delta^1}$}
\label{connecting}
In this section we compute the kernel and cokernel
of the connecting maps $\delta^0$ and $\delta^1$.
Using these computations, we prove Theorem \ref{cohomology}.
\subsection{$\pmb{\delta^0:\ker g\to\ker h}$}
By Proposition \ref{filtration}, Corollary
\ref{gcor}, and Proposition \ref{hprop}(a), the source
of $\delta^0$ is
\[
\ker g =\pi_0TMF_{(3)}\oplus(\pi_*TMF_{(3)})_T
\]
and its target is
\[
\ker h =\Z_{(3)}\{a_{-i,i}:i\geq1\}
\subset B_0=\pi_0TMF_0(2)_{(3)}.
\]
For degree reasons, 
$(\pi_*TMF_{(3)})_T\subset\ker\delta^0$,
so it suffices to study the effect of $\delta^0$
on elements of $\pi_0TMF_{(3)}$.

The computation in (\cite{Larson:ANSS}, Eq.\ (24))
implies that for $v\geq0$,
\begin{equation}\label{delzerojmf}
\delta^0(C_v^0)
=2^{8v}
\sum_{i=1}^v\left(
\binom{3v}{2v+i}4^{-i}-\binom{3v}{2v-i}4^i
\right)a_{-i,i}
-2^{8v}\sum_{i=v+1}^{2v}\binom{3v}{2v-i}4^i a_{-i,i}.
\end{equation}  
Thus, with respect to the ordered bases
$\{C_0^0,3C_1^0,3C_2^0,
C_3^0,3C_4^0,3C_5^0,\ldots\}$
(see Corollary \ref{TMFZERO})
and 
$\{a_{-1,1}, a_{-2,2},\ldots\}$, 
\begin{equation}\label{deltazeromatrix}
\delta^0\bigg|_{\pi_0TMF_{(3)}}=\left[
\begin{array}{ccccc}
0 & * & * & * & \, \\
\vdots & 3u_1 & * & * & \,  \\
\, & 0 & * & * & \, \\
\, & \vdots & 3u_2 & * & \,  \\
\, & \, & 0 & * & \, \\
\, & \, & \vdots & u_3 & \,  \\
\, & \, & \, & 0 & \ddots 
\end{array}\right]
\end{equation}
where $u_k=-2^{12k}\in\ZLOC^{\times}$ for $k\geq1$
(cf. \cite{Larson:ANSS}, Eq.\ (25)).
The form of the matrix in Eq.\ \eqref{deltazeromatrix} 
shows that we may take
\[
\{a_{-1,1},\delta^0(C_1^0),a_{-3,3},
\delta^0(C_2^0),a_{-5,5},\delta^0(C_3^0),\ldots
\}
\]
as an alternative ordered basis for $\ker h$,
thereby proving the following proposition.
\begin{prop}\label{deltazeroresults}
$\ker\delta^0=\ZLOC\{1\}\oplus
(\pi_*TMF_{(3)})_T$, and
\[
\coker\delta^0=
\ZLOC\{a_{-i,i}:i\geq1\text{ odd}\}
\oplus\Z/(3)\{\delta^0(C_v^0):v\geq1, v
\text{ not a multiple of 3}\}.
\]
\end{prop}

\subsection{$\pmb{\delta^1:\coker g\to\coker h}$}
\label{deltaonethree}
By Proposition \ref{filtration}, Corollary
\ref{gcor}, and Proposition \ref{hprop}(b), the source
of $\delta^1$ is
\[
\coker g=\pi_*TMF_{(3)}/\sim_{\coker g}
\]
and its target is
\[
\coker h=\bigoplus_{i<j\in\Z}\left(
\Z/3^{\nu_3(i+j)+1}\{a_{i,j}\}\oplus
\Z/3^{\nu_3(2i+2j+1)+1}\{b_{i,j}\}
\right).
\]
\begin{lem}\label{deltaoneintro}
\begin{enumerate}[(a)]
\item $(\pi_*TMF_{(3)})_T\subset\ker\delta^1$
\item $\ker
\delta^1\bigg|_{\pi_0TMF_{(3)}}=
\ker\delta^0\bigg|_{\pi_0TMF_{(3)}}$, and similarly
for the cokernel.
\end{enumerate}
\end{lem}
\begin{proof}
Part (a) follows from the facts that $\delta^1$
is induced by $-\phi_f$, $\phi_f:r\mapsto0$ (see
Definition \ref{themaps}), and
$\alpha$ (resp.\ $\beta$) is represented by $r$
(resp.\ $r^2\otimes r-r\otimes r^2$) in the cobar
resolution (see Remark \ref{cobar}).

The computation in (\cite{Larson:ANSS}, Eq.\ 26) implies
that for $v\geq0$,
\begin{equation}\label{delonejmf}
\delta^1(C_v^0)
=\frac12\delta^0(C_v^0)
\end{equation}
which proves (b).
\end{proof}
By Lemma \ref{zerolinedecomp}, it remains
to study $\delta^1\bigg|_{W^{\epsilon,m}}$
for $(\epsilon,m)\neq(0,0)$, and we shall
do so for the remainder of this subsection.
First, however, we establish some convenient
notational conventions.
\begin{defn}\label{paradigm}
\begin{enumerate}[(a)]
\item
If $x=\gamma_v C_v^m
\in\BMF^{0,m}$
(resp.\ $y=\theta_v D_v^m\in\BMF^{1,m}$),
the $A_w^m$ term of
$\delta^1(x)$
(resp.\ the $B_w^m$ term of 
$\delta^1(y)$)
with the {\em greatest} subscript $w$
will be denoted the {\em leading term}, 
and we will refer to the remaining
terms as {\em lower order terms}
(see Lemma \ref{deltaonesetup}).
\item The symbol $\doteq$ will denote equality
up to multiplication by a unit in $\ZLOC$.
\item If $M$ is a matrix with columns
$M_1,\ldots,M_v$ and $N$ is a matrix
with (possibly infinitely many) columns
$N_1,N_2,\ldots$, then $M\boxplus N$
will denote the matrix with columns
\[
M_1,\ldots,M_v,N_1,N_2,\ldots.
\]
\end{enumerate}
\end{defn} 
The following lemma sets up
our subsequent computations with $\delta^1$.
\begin{lem}\label{deltaonesetup}
\begin{enumerate}[(a)]
\item The source of $\delta^1\bigg|_{W^{0,m}}$
has basis $\mathcal B_{TMF}^{0,m}=
\{\gamma_0C_0^m,\gamma_1C_1^m,\ldots\}$
and its target has basis
$\{A_0^m,A_1^m,\ldots\}\subset\coker h$.
Moreover, if $m<0$,
\begin{equation}\label{deltaonefirst}
\delta^1(C^m_v)\doteq A^m_{\left\lfloor\frac{m-1}2
\right\rfloor-2\LMAX+2v}+\HOT
\end{equation}
and if $m>0$,
\begin{equation}\label{deltaonesecond}
\delta^1(C^m_v)\doteq
\begin{cases}
A^m_{\left\lfloor\frac{m-1}2\right\rfloor-\LMAX+v}
+\HOT, & 0\leq v<\LMAX\\
0, & v=\LMAX\\
A^m_{\left\lfloor\frac{m-1}2\right\rfloor-2\LMAX+2v}
+\HOT, & v>\LMAX.
\end{cases}
\end{equation}
\item The source of $\delta^1\bigg|_{W^{1,m}}$
has basis $\mathcal B_{TMF}^{1,m}=\{
\theta_0D_0^m,\theta_1D_1^m,\ldots
\}$ and its target has basis
$\{
B_0^m,B_1^m,\ldots
\}\subset\coker h$.  Moreover, if $m\leq0$,
\begin{equation}\label{deltaonethird}
\delta^1(D^{m}_v)\doteq B^{m}_{\left\lfloor\frac{m-1}2
\right\rfloor-2\LMAXB+2v}+\HOT
\end{equation}
and if $m>0$,
\begin{equation}\label{deltaonefourth}
\delta^1(D^{m}_v)\doteq
\begin{cases}
B^{m}_{\left\lfloor\frac{m-1}2
\right\rfloor-\LMAXB+v}+\HOT, & 0\leq v<\LMAXB \\
*, & v=\LMAXB \\
B^{m}_{\left\lfloor\frac{m-1}2
\right\rfloor-2\LMAX+2v}+\HOT, & v>\LMAXB
\end{cases}
\end{equation}
where $*=0$ if $m\nequiv13\, (27)$ and
$*\neq0$ otherwise.
\end{enumerate}
\end{lem}
\begin{proof}
Equations \eqref{deltaonefirst},
\eqref{deltaonesecond}, \eqref{deltaonethird},
and \eqref{deltaonefourth} follow from
the proof of (\cite{Larson:ANSS}, Proposition 9)
and the results cited therein,
{\em mutatis mutandis}. Lemma \ref{bdirectsum},
Proposition
\ref{TMFBASES}, and 
Remark \ref{degreecounting} imply the remaining
statements.
\end{proof}
Guided by Lemma \ref{deltaonesetup}, we shall now
divide our remaining computations of $\delta^1$ into cases
depending on the values of $\epsilon$ and $m$.
In each case we will find matrix representations
of $\delta^1$, just as we did
with $\delta^0$ in Eq.\ \eqref{deltazeromatrix}.
\begin{rmk}
In Eq.\ \eqref{deltaonefourth} when $m\equiv13\, (27)$,
the quantity $*=\delta^1(D^m_{\LMAXB})$ is a 
{\em nontrivial} linear combination of the generators $B_v^m$,
which is precisely the reason for the undetermined
submodule $U^*\subset\,_{BK}E_2^{1,*}\QTWO$ in
Theorem \ref{main}.  This is explained further below
in Case 5.  An analogous phenomenon occured in
our computation in \cite{Larson:ANSS}: see, e.g., 
(\cite{Larson:ANSS}, Remark 2).
\end{rmk}

\begin{proof}[Case 1: $\epsilon=0$, $m<0$.]\let\qed\relax
By Eq.\ \eqref{deltaonefirst}, the matrix representation
with respect to the bases given in Lemma \ref{deltaonesetup}(a)
is
\begin{equation}\label{MATRIX1}
\delta^1\bigg|_{W^{0,m}}=
\left[
\begin{array}{cccc}
\vdots & \vdots & \, & \,  \\
\gamma_0u_0 & * & \, & \,  \\
0 & * & \vdots & \,  \\
\vdots & \gamma_1u_1 & * & \,  \\
\, & 0 & * & \,  \\
\, & \vdots & \gamma_2u_2 & \,  \\
\, & \, & 0 & \ddots  
\end{array}\right]
\end{equation}
where $u_0,u_1,u_2,\ldots\in\ZLOC^{\times}$
and $u_0$ is in the row corresponding to
$A^m_{\floor{\frac{m-1}2}-2\LMAX}$.  Motivated by the 
structure of this matrix,
we may construct an alternative 
ordered basis for the image, namely
\begin{align*}
\{A_0^m,\ldots,A_{\floor{\frac{m-1}2}-2\LMAX-1}^m,
&\delta^1(C_0^m),A_{\floor{\frac{m-1}2}-2\LMAX+1}^m,\\
&\delta^1(C_1^m),A_{\floor{\frac{m-1}2}-2\LMAX+3}^m,\\
&\delta^1(C_2^m),A_{\floor{\frac{m-1}2}-2\LMAX+5}^m,\ldots
\}.
\end{align*}
From this, we may deduce the kernel and cokernel.
\begin{lem}\label{LEMMAONE}
For $m<0$,
\[
\ker\delta^1\bigg|_{W^{0,m}}=
\begin{cases}
\Z/(3)\{3C_1^m,3C_2^m,3C_4^m,3C_5^m,3C_7^m,3C_8^m,\ldots\}, & 
m\equiv0,1,2\ (9),\\
\Z/(3)\{3C_0^m;3C_2^m,3C_3^m,3C_5^m,3C_6^m,3C_8^m,3C_9^m,\ldots\}, & 
m\equiv3,4,5\ (9),\\
\Z/(3)\{3C_0^m,3C_1^m,3C_3^m,3C_4^m,3C_6^m,3C_7^m,\ldots\}, &
\text{otherwise}
\end{cases}
\]
and 
\begin{align*}
\coker\delta^1\bigg|_{W^{0,m}}=
\Z/(3)\{\mathcal A\}&\oplus
\Z/(3^{\nu_3(m)+1})
\left\{
A^m_0,\ldots,A^m_{\left\lfloor\frac{m-1}2
\right\rfloor-2\LMAX-1}\right\}\\
&\oplus\Z/(3^{\nu_3(m)+1})\left\{
A^m_{\left\lfloor\frac{m-1}2
\right\rfloor-2\LMAX+i}:i\geq1,\rm{\, odd}
\right\}
\end{align*}
where 
\[
\mathcal A:=\begin{cases}
\{\delta^1(C_1^m),\delta^1(C_2^m),\delta^1(C_4^m),
\delta^1(C_5^m),\delta^1(C_7^m),
\delta^1(C_8^m),\ldots\},&m\equiv0,1,2\ (9),\\
\{\delta^1(C_0^m);\delta^1(C_2^m),\delta^1(C_3^m),
\delta^1(C_5^m),\delta^1(C_6^m),
\delta^1(C_8^m),\delta^1(C_9^m),\ldots\},
&m\equiv3,4,5\ (9),\\
\{\delta^1(C_0^m),\delta^1(C_1^m),\delta^1(C_3^m),
\delta^1(C_4^m),\delta^1(C_6^m),
\delta^1(C_7^m),\ldots\},&\text{otherwise}.
\end{cases}
\]
\end{lem}
\end{proof}
\begin{proof}[Case 2: $\epsilon=0$, $m>0$.]\let\qed\relax
By Eq.\ \eqref{deltaonesecond}, the matrix
representation in this case is
\begin{equation}\label{MATRIX2}
\delta^1\bigg|_{W^{0,m}}=
\left[
\begin{array}{ccccc}
\vdots & \vdots & \, & \, & \,\\
\gamma_0u_0 & * & \, & \, & \,\\
0 & \gamma_1u_1 & \, & \vdots & \,\\
\vdots & 0 & \ddots & * & {\pmb\vdots}\\
\, & \vdots & \, & \gamma_yu_y & {\pmb 0}\\
\, & \, & \, & 0 & {\pmb 0}\\
\, & \, & \, & \vdots & {\pmb\vdots}\\
\end{array}
\right]\boxplus
\left[
\begin{array}{cccc}
\vdots & \vdots & \,& \,\\
\gamma_{y+2}u_{y+2} & * & \, & \,\\
0 & * & \vdots & \,\\
\vdots & \gamma_{y+3}u_{y+3} & *& \,\\
\, & 0 & * & \, \\
\, & \vdots & \gamma_{y+4}u_{y+4} & \, \\
\, & \, & 0 & \ddots
\end{array}
\right]
\end{equation}
where 
the $u_i$ are units in $\ZLOC$.
Here, $u_0$ is in the row
corresponding to $A^m_{\left\lfloor\frac{m-1}2\right\rfloor
-\LMAX}$, $u_y$
is in the row corresponding to $A^m_{\left\lfloor\frac{m-1}2\right\rfloor
-1}$, $u_{y+2}$ is in the row corresponding to
$A^m_{\left\lfloor\frac{m-1}2\right\rfloor+2}$,
and the zero column in bold corresponds to 
$C^m_{\LMAX}$. As in the previous case, we may 
deduce the kernel and cokernel by 
constructing
an alternative basis for the image; this time it is
\begin{align*}
\{A_0^m,\ldots,A^m_{\floor{\frac{m-1}2}-\LMAX-1};
\delta^1(C_0^m),\ldots,\delta^1(C^m_{\LMAX-1});
A_{\floor{\frac{m-1}2}}; 
& A_{\floor{\frac{m-1}2}+1},\delta^1(C^m_{\LMAX+1}),\\
& A_{\floor{\frac{m-1}2}+3},\delta^1(C^m_{\LMAX+2}),\\
& A_{\floor{\frac{m-1}2}+5},\delta^1(C^m_{\LMAX+3}),\ldots
\}.
\end{align*}
\begin{lem}\label{LEMMATWO}
For $m<0$,
$\ker\delta^1\bigg|_{W^{0,m}}=
\Z/(3^{\nu_3(m)+1})\{C^m_{\LMAX}\}\oplus K$ where
\[
K:=
\begin{cases}
\Z/(3)\{3C_1^m,3C_2^m,3C_4^m,3C_5^m,3C_7^m,3C_8^m,\ldots\}, & 
m\equiv0,1,2\ (9),\\
\Z/(3)\{3C_0^m;3C_2^m,3C_3^m,3C_5^m,3C_6^m,3C_8^m,3C_9^m,\ldots\}, & 
m\equiv3,4,5\ (9),\\
\Z/(3)\{3C_0^m,3C_1^m,3C_3^m,3C_4^m,3C_6^m,3C_7^m,\ldots\}, &
\text{otherwise}
\end{cases}
\]
and
\begin{align*}
\coker\delta^1\bigg|_{W^{0,m}}=
\Z/(3)\{\mathcal A\}&\oplus
\Z/(3^{\nu_3(m)+1})
\left\{
A^m_0,\ldots,A^m_{\left\lfloor\frac{m-1}2
\right\rfloor-\LMAX-1};A^m_{\floor{\frac{m-1}2}}\right\}\\
&\oplus\Z/(3^{\nu_3(m)+1})\left\{
A^m_{\left\lfloor\frac{m-1}2
\right\rfloor+i}:i\geq1,\rm{\, odd}
\right\}
\end{align*}
where $\mathcal A$ is defined as in Lemma 
\ref{LEMMAONE}.
\end{lem}
\end{proof}
\begin{proof}[Case 3: $\epsilon=1$, $m\leq0$.]\let\qed\relax
By Eq.\ \eqref{deltaonethird}, the matrix representation with
respect to the bases given in Lemma \ref{deltaonesetup}(b) is
$\delta^1\bigg|_{W_{1,m}}$
identical in form to the one in (\ref{MATRIX1}), except
with $\gamma_i$ replaced by $\theta_i$
everywhere.
In this case the unit $u_0$ 
appears in the row corresponding to
$B^{m}_{\left\lfloor\frac{m-1}2\right\rfloor-2\LMAXB}$.  
Therefore, we may argue as in Case 1 to obtain
the following lemma.
\begin{lem}\label{LEMMATHREE}
For $m\leq0$,
\[
\ker\delta^1\bigg|_{W^{1,m}}=
\begin{cases}
\Z/(3)\{3D_1^m,3D_2^m,3D_4^m,3D_5^m,3D_7^m,3D_8^m,\ldots\}, & 
m\equiv1,2,3\ (9),\\
\Z/(3)\{3D_0^m;3D_2^m,3D_3^m,3D_5^m,3D_6^m,3D_8^m,3D_9^m,\ldots\}, & 
m\equiv4,5,6\ (9),\\
\Z/(3)\{3D_0^m,3D_1^m,3D_3^m,3D_4^m,3D_6^m,3D_7^m,\ldots\}, &
\text{otherwise}
\end{cases}
\]
and
\begin{align*}
\coker\delta^1\bigg|_{W^{1,m}}=
\Z/(3)\{\mathcal B\}&\oplus
\Z/(3^{\nu_3(2m+1)+1})
\left\{
B^m_0,\ldots,B^m_{\left\lfloor\frac{m-1}2
\right\rfloor-2\LMAX-1}\right\}\\
&\oplus\Z/(3^{\nu_3(2m+1)+1})\left\{
B^m_{\left\lfloor\frac{m-1}2
\right\rfloor-2\LMAX+i}:i\geq1,\rm{\, odd}
\right\}
\end{align*}
where
\[
\mathcal B:=\begin{cases}
\{\delta^1(D_1^m),\delta^1(D_2^m),\delta^1(D_4^m),
\delta^1(D_5^m),\delta^1(D_7^m),
\delta^1(D_8^m),\ldots\},&m\equiv1,2,3\ (9),\\
\{\delta^1(D_0^m);\delta^1(D_2^m),\delta^1(D_3^m),
\delta^1(D_5^m),\delta^1(D_6^m),
\delta^1(D_8^m),\delta^1(D_9^m),\ldots\},
&m\equiv4,5,6\ (9),\\
\{\delta^1(D_0^m),\delta^1(D_1^m),\delta^1(D_3^m),
\delta^1(D_4^m),\delta^1(D_6^m),
\delta^1(D_7^m),\ldots\},&\text{otherwise}.
\end{cases}
\]
\end{lem}
\end{proof}
\begin{proof}[Case 4: $\epsilon=1$, $m>0$, $m\nequiv13\, (27)$.]
\let\qed\relax
By Eq.\ \eqref{deltaonefourth},
the matrix representation in 
this case is identical in form
to the one in (\ref{MATRIX2}), except with
$\gamma_i$ replaced by $\theta_i$ everywhere.
Moreover,
$u_0$ is in the row corresponding to 
$B^{m}_{\left\lfloor\frac{m-1}2\right\rfloor
-\LMAXB}$, $u_y$
is in the row corresponding to $B^{m}_{\left\lfloor\frac{m-1}2\right\rfloor
-1}$, and $u_{y+2}$ 
is in the row corresponding to
$B^{m}_{\left\lfloor\frac{m-1}2\right\rfloor+2}$.
Therefore, we may argue as in Case 2 to obtain
the following lemma.
\begin{lem}\label{LEMMAFOUR}
For $m>0$ and $m\nequiv13\, (27)$, $\ker\delta^1\bigg|_{W^{1,m}}=
\Z/(3^{\nu_3(2m+1)+1})\{D^m_{\LMAXB}\}\oplus K'$
where
\[
K':=
\begin{cases}
\Z/(3)\{3D_1^m,3D_2^m,3D_4^m,3D_5^m,3D_7^m,3D_8^m,\ldots\}, & 
m\equiv1,2,3\ (9),\\
\Z/(3)\{3D_0^m;3D_2^m,3D_3^m,3D_5^m,3D_6^m,3D_8^m,3D_9^m,\ldots\}, & 
m\equiv4,5,6\ (9),\\
\Z/(3)\{3D_0^m,3D_1^m,3D_3^m,3D_4^m,3D_6^m,3D_7^m,\ldots\}, &
\text{otherwise}
\end{cases}
\]
and 
\begin{align*}
\coker\delta^1\bigg|_{W^{1,m}}=
\Z/(3)\{\mathcal B\}&\oplus
\Z/(3^{\nu_3(2m+1)+1})
\left\{
B^m_0,\ldots,B^m_{\left\lfloor\frac{m-1}2
\right\rfloor-\LMAX-1};B^m_{\floor{\frac{m-1}2}}\right\}\\
&\oplus\Z/(3^{\nu_3(2m+1)+1})\left\{
B^m_{\left\lfloor\frac{m-1}2
\right\rfloor+i}:i\geq1,\rm{\, odd}
\right\}
\end{align*}
where $\mathcal B$ is defined as in Lemma \ref{LEMMATHREE}.
\end{lem}
\end{proof}
\begin{proof}[Case 5: $\epsilon=1$, $m>0$, $m\equiv13\, (27)$.]
\let\qed\relax
In this final case, Eq.\ \eqref{deltaonefourth}
implies that the matrix
representation is identical in form to
the one in Case 4, 
except that the column
in bold is no longer a 
column of zeros.  Rather, 
by (\cite{Larson:ANSS}, Lemma 12(a)) it has at least one
nonzero entry in and above the row
containing $u_y$. 
This yields the following lemma.
\begin{lem}\label{LEMMAFIVE}
Suppose $m>0$ and $m\equiv13\, (27)$.
\begin{enumerate}[(a)]
\item The cokernel of $\delta^1\bigg|_{W^{1,m}}$ 
has the same
presentation as in Case 4,
but now including a single nontrivial relation
\begin{equation}\label{maintheoremrelation}
\delta^1(D^m_{\LMAXB})=0
\end{equation}
among its generators.
\item The kernel of $\delta^1\bigg|_{W^{1,m}}$
has, as a direct summand, 
\[
K'':=\Z/(3)\{3D^m_{\LMAXB+1},3D^m_{\LMAXB+2},
3D^m_{\LMAXB+4},3D^m_{\LMAXB+5},
3D^m_{\LMAXB+7},3D^m_{\LMAXB+8},
\ldots
\}.
\]
\end{enumerate}
\end{lem}
\begin{defn}\label{indeterminacy}
The graded module $U^*$ in Theorem \ref{main}
is determined by the direct sum decomposition
\[
\ker\delta^1\bigg|_{W^{1,m}}=K''\oplus U^{4m+2}.
\]
\end{defn}
We do not have sufficient knowledge
of the coefficients in the columns of the
matrix in Eq.\ \eqref{MATRIX2}
corresponding to $D^m_0,\ldots,D^m_{\LMAXB}$
(e.g., their 3-divisibility)
to explicitly compute $U$.
\end{proof}
\begin{proof}[Proof of Theorem \ref{cohomology}]
Propositions \ref{filtration} and \ref{deltazeroresults}
yield the result for 
$H^0\mathcal G_*^{\bullet}$.
Proposition \ref{deltazeroresults} 
also implies that
$\coker\delta^0$ is 
concentrated in $t$-degree zero,
while Proposition \ref{deltaoneintro} implies
that $\ker\delta^1$ is a free $\ZLOC$-module
of rank 1 generated by $1_{\pi_*TMF_{(3)}}$.
Thus, the short exact sequence
(\ref{hone}) splits as claimed in Proposition 
\ref{filtration}, which means the result
for $H^1C^*$ follows from
Proposition \ref{deltazeroresults} combined with Lemmas
\ref{LEMMAONE}, \ref{LEMMATWO}, \ref{LEMMATHREE},
\ref{LEMMAFOUR}, and \ref{LEMMAFIVE}(b).
Finally, the result for $H^2C^*$
follows from Proposition \ref{filtration},
Lemmas \ref{LEMMAONE} --
\ref{LEMMAFOUR} and \ref{LEMMAFIVE}(a),
and Definition \ref{indeterminacy}.
\end{proof}

\section{Higher differentials in the BKSS}
\label{higherdiffs}
In this final section we briefly examine
the anatomy of the differentials on
the $E_r$-page of the BKSS for $\QTWO$
for $r\geq2$.  

A schematic diagram of $\,_{BK}E_2^{s,t}\QTWO$
is shown in Figure \ref{ETWO},
with $s$ along the vertical axis
and $t-s$ along the horizontal axis.
Included in the diagram is an
example of a possibly nontrivial $d_2$-differential
$\,_{BK}E_2^{0,8}\QTWO\to
\,_{BK}E_2^{2,7}\QTWO$.  In fact, since
the only nontrivial rows in the diagram
are those corresponding to $s=0$, $s=1$,
and $s=2$ by Theorem \ref{main}, the only
possibly nontrivial $d_2$-differentials
are those that map from the 0-line
to the 2-line.
\begin{figure}[h!]
\centering
\begin{sseq}[entrysize=8mm,ylabelstep=2,xlabelstep=4]{-4...12}
{0...4}
\ssdrop{\longleftarrow\,_{BK}E^{0,*}_2\QTWO\longrightarrow}
\ssmove{0}{1}
\ssdrop{\longleftarrow\,_{BK}E^{1,*}_2\QTWO\longrightarrow}
\ssmove{0}{1}
\ssdrop{\longleftarrow\,_{BK}E^{2,*}_2\QTWO\longrightarrow}
\ssmove{0}{1}
\ssdrop{\longleftarrow0\longrightarrow}
\ssmove{0}{1}
\ssdrop{\vdots}
\ssmoveto{8}{0}
\ssdropbull
\ssmove{-3}{2}
\ssdropbull
\ssstroke[arrowto]
\end{sseq}
\caption{The $E_2$-term of
the BKSS for $\QTWO$.}
\label{ETWO}
\end{figure}
Moreover, by sparseness, 
all $d_2$-differentials mapping into or
out of the 1-line are identically zero,
as are all $d_r$-differentials
for $r\geq3$.
The following theorem summarizes the situation.
\begin{thm}
In the BKSS for $\QTWO$,
\begin{enumerate}[(a)]
\item The only possibly nontrivial $d_2$-differentials
have the form
\[
d_2:\,_{BK}E_2^{0,t}\QTWO\to
\,_{BK}E_2^{2,t-1}\QTWO,
\]
\item All elements in $\,_{BK}E_2^{1,*}\QTWO$
are permanent cycles,
\item The spectral sequence collapses
at $E_3$, i.e., $\,_{BK}E_3^{s,t}\QTWO=
\,_{BK}E_{\infty}^{s,t}\QTWO$.
\end{enumerate}
\end{thm}

\footnotesize
\bibliographystyle{plain}
\bibliography{math}

\end{document}